\documentclass[conference]{IEEEtran}
\IEEEoverridecommandlockouts
 
\usepackage{hyperref}

\usepackage{cite}
\usepackage{amsmath,amssymb,amsfonts}
\usepackage{algorithmic}
\usepackage{graphicx}
\usepackage{textcomp}
\usepackage{xcolor}
\usepackage{amssymb}
\usepackage{mathtools}
\usepackage{enumitem}
\usepackage{amsthm}

\newtheorem{thm}{Theorem}[section]
\newtheorem{lem}[thm]{Lemma}

\def\BibTeX{{\rm B\kern-.05em{\sc i\kern-.025em b}\kern-.08em
    T\kern-.1667em\lower.7ex\hbox{E}\kern-.125emX}}
\begin{document}

\title{Wasserstein Distributionally Robust Regret-Optimal Control under Partial Observability\\
{\footnotesize \textsuperscript{}}
\thanks{The authors are affiliated with the Department of Electrical Engineering at Caltech. Emails: \{jhajar,tkargin,hassibi\}@caltech.edu.}}

\author{\IEEEauthorblockN{Joudi Hajar}

\and
\IEEEauthorblockN{Taylan Kargin}

\and
\IEEEauthorblockN{Babak Hassibi}

}
\maketitle
\thispagestyle{plain}
\pagestyle{plain}
\begin{abstract}
This paper presents a framework for Wasserstein distributionally robust (DR) regret-optimal (RO) control in the context of partially observable systems. DR-RO control considers the regret in LQR cost between a causal and non-causal controller and aims to minimize the worst-case regret over all disturbances whose probability distribution is within a certain Wasserstein-2 ball of a nominal distribution. Our work builds upon the full-information DR-RO problem that was introduced and solved in Yan et al., 2023 \cite{DRORO}, and extends it to handle partial observability and measurement-feedback (MF). We solve the finite horizon partially observable DR-RO and show that it reduces to a tractable semi-definite program whose size is proportional to the time horizon.  Through simulations, the effectiveness and performance of the framework are demonstrated, showcasing its practical relevance to real-world control systems. The proposed approach enables robust control decisions, enhances system performance in uncertain and partially observable environments, and provides resilience against measurement noise and model discrepancies.
\end{abstract}


\section{Introduction}

Regret-optimal control~\cite{sabag2021regret,goel2023regret,didier2022system,martin2022safe,DRORO}, is a new approach in control theory that focuses on minimizing the regret associated with control actions in uncertain systems. The regret measures the cumulative difference between the performance achieved by a causal control policy and the performance achieved by an optimal policy that could have been chosen in hindsight. In regret-optimal control, the worst-case regret over all $\ell_2$-norm-bounded disturbance sequences is minimized. 

Distributionally robust control~\cite{yang2020wasserstein,tacskesen2023distributionally,zhong2023nonlinear,DRORO}, on the other hand, addresses uncertainty in system dynamics and disturbances by considering a set of plausible probability distributions rather than relying on a single distribution as in LQG control, or on a worst-case disturbance, such as in $H_\infty$ or RO control. This approach seeks to find control policies that perform well across all possible distributions within the uncertainty set, thereby providing robustness against model uncertainties and ensuring system performance in various scenarios. The size of the uncertainty set allows one to control the amount of desired robustness so that, unlike $H_\infty$ controllers, say, the controller is not overly conservative. The uncertainty set is most often taken to be the set of disturbances whose distributions are within a given Wasserstein-2 distance of the nominal disturbance distribution. The reason is that, for quadratic costs, the supremum of the expected cost over a Wasserstein ball reduces to a tractable semi-definite program (SDP). 




The current paper considers and extends the framework introduced in \cite{DRORO} that applied distributionally robust (DR) control to the regret-optimal (RO) setting. In the full-information finite-horizon setting, the authors of \cite{DRORO} reduce the DR-RO problem to a tractable SDP. In this paper, we extend the results of \cite{DRORO} to partially observable systems where, unlike the full-information setting, the controller does not have access to the system state. Instead, it only has access to partial information obtained through noisy measurements. This is often called the measurement feedback (MF) problem. Of course, the solution to the measurement feedback problem in LQG and $H_\infty$ control is classical. The measurement-feedback setting for DR control has been studied in ~\cite{tacskesen2023distributionally},~\cite{hakobyan2022wasserstein}, and for RO control in 
~\cite{ROMF}.  

In the finite-horizon case, we reduce the DR-RO control problem with measurement feedback to an SDP similar to the full-information case studied in \cite{DRORO}. Furthermore,  we validate the effectiveness and performance of our approach through simulations, showcasing its applicability in real-world control systems.



The organization of the paper is as follows. In section~\ref{sec::prelim}, we review the LQG and regret optimal control formulation in the measurement-feedback setting. In section~\ref{sec::pb}, we present the distributionally robust regret-optimal with measurement feedback (DR-RO-MF) problem formulation,  in section~\ref{sec::cvx} we reformulate the problem as a tractable SDP, and in section~\ref{sec::exp} we show numerical results for controlling the flight of a Boeing 747~\cite{boeing}.

\section{Preliminaries}\label{sec::prelim}
\subsection{Notations}
$\mathbb{R}$ denotes the set of real numbers, $\mathbb{N}$ is the set of natural numbers, $\| \cdot\|$ is the 2-norm, $\mathbb{E}_{(\cdot)}$ is the expectation over $(\cdot)$, $\mathcal{M}(\cdot)$ is the set of probability distributions over $(\cdot)$ and $\operatorname{Tr}$ denotes the trace.
\subsection{A Linear Dynamical System}
We consider the following state-space model of a discrete-time, linear time-invariant (LTI) dynamical system:
\begin{equation} \label{eq::ss} 
\begin{aligned}
        x_{t+1}&=Ax_t+Bu_t+w_t, \\
    y_t&=Cx_t+v_t. 
\end{aligned}
\end{equation}
Here, $x_t\in \mathbb{R}^n$ represents the state of the system, $u_t\in \mathbb{R}^m$ is the control input, $w_t \in \mathbb{R}^n$ is the process noise, while $y_t \in \mathbb{R}^p$ represents the noisy state measurements that the controller has access to, and $v_t \in \mathbb{R}^p$ is the measurement noise.  The sequences \{$w_i$\} and \{$v_i$\} are considered to be randomly distributed according to an unknown joint probability measure $P$ which lies in a specified compact ambiguity set, $\cal P$. For simplicity, we take $x_0$ to be zero.

In the rest of this paper, we adopt an operator form representation of the system dynamics~\eqref{eq::ss}. To this end, assume a horizon of $N\in \mathbb{N} $, and let us define 
\[ x \coloneqq \left[ \begin{array}{c} x_0 \\ x_1 \\ \vdots \\ x_{N-1} \end{array} \right] \in {\mathbb R}^{Nn}~~~,~~~
u \coloneqq \left[ \begin{array}{c} u_0 \\ u_1 \\ \vdots \\ u_{N-1} \end{array} \right] \in {\mathbb R}^{Nm}\]
and similarly for $y\in{\mathbb R}^{Np}$, $w\in{\mathbb R}^{Nn}$, and $v\in{\mathbb R}^{Np}$. Using these definitions, we can represent the system dynamics~\eqref{eq::ss} equivalently in operator form as
\begin{equation} \label{eq::of}
\begin{aligned}
    x&=Fu+Gw, \\
    y&=Ju+Lw+v,
\end{aligned}
\end{equation}
where $F\in{\mathbb R}^{Nn\times Nm}$, $G\in{\mathbb R}^{Nn\times Nn}$, $J\in{\mathbb R}^{Np\times Nm}$, and $L\in{\mathbb R}^{Np\times Nn}$ are strictly causal time-invariant operators (i.e, strictly lower triangular block Toeplitz matrices) corresponding to the dynamics~\eqref{eq::ss}.

We consider the Linear-Quadratic Gaussian (LQG) cost given as
\begin{equation}\label{eq::LQRcost1}
    J(u,w,v)\coloneqq x^TQx+u^TRu
\end{equation}
where $Q, R\succ0$ are positive definite matrices of the appropriate dimensions. In order to simplify the notation, we redefine $x$ and $u$ as $x\leftarrow Q^{\frac{1}{2}}x$, and $u\leftarrow R^{\frac{1}{2}}u$, so that~\eqref{eq::LQRcost1} becomes
\begin{equation} \label{eq::LQRcost2}
    J(u,w,v)=\|x\|^2+\|u\|^2.
\end{equation}

\subsection{Controller Design}
We consider a linear controller that has only access to the measurements: 
\begin{equation}
    u=Ky, \quad K\in \mathcal{K},
\end{equation}
where $\mathcal{K}\subseteq \mathbb R^{Nm\times Np}$ is the space of causal (i.e., lower triangular) matrices.
Then, the closed-loop state measurement becomes
\begin{equation}
    y=(I-JK)^{-1}(Lw+v).
\end{equation}
As in \cite{ROMF}, let 
\begin{equation}\label{eq::E1}
    E=K(I-JK)^{-1},
\end{equation}
be the Youla parametrization, so that 
\begin{equation}\label{eq::controller}
    K=(I+EJ)^{-1}E.
\end{equation} The closed-loop LQG cost~\eqref{eq::LQRcost2} can then be written as:
\begin{equation}
    J(K,w,v)= 
\begin{bmatrix}
    w^T & v^T
\end{bmatrix} T_K^T T_K \begin{bmatrix}
    w \\
    v
\end{bmatrix},
\end{equation}
where $T_K$ is the transfer operator associated with $K$ that maps the disturbance sequences $\begin{bmatrix}
    w\\v
\end{bmatrix}$ to the state and control sequences $\begin{bmatrix}
    x\\u
\end{bmatrix}$:
\begin{equation}
   T_K\coloneqq \begin{bmatrix}FEL+G & FE \\ EL & E
   \end{bmatrix}. 
\end{equation}

\subsection{Regret-Optimal Control with Measurement-Feedback}\label{sec::ROMF}
Given a noncausal controller $K_0 \!\in\! \cal K$, we define the regret as:
\begin{align}\label{eq::regret}
    R(K,w,v)&\coloneqq J(K,w,v)- J(K_0,w,v),\\
    &= \begin{bmatrix}
    w^T & v^T
\end{bmatrix} (T_K^T T_K-T_{K_0}^T T_{K_0})\begin{bmatrix}
    w \\
    v\\
\end{bmatrix},
\end{align}
which measures the excess cost that a causal controller suffers by not knowing the future. In other terms, regret is the difference between the cost accumulated by a causal controller and the cost accumulated by a benchmark noncausal controller that knows the complete disturbance trajectory. The problem of minimizing regret in the measurement-feedback setting is referred to as (RO-MF) and is formulated as:
\begin{equation}\label{eq::ROMF}
\inf_{K\in \mathcal{K}} \sup_{\substack{ w,v}}\frac{R(K,w,v)}{\|w\|^2+ \|v\|^2},
\end{equation} 
which is solved suboptimally by reducing it to a level-1 suboptimal Nehari problem \cite{ROMF}.

\section{Distributionally Robust Regret-Optimal Control}\label{sec::pb}
In this section, we introduce the {\bf distributionally robust regret-optimal} (DR-RO) control problem {\bf with  measurement feedback}, which we refer to as {\bf DR-RO-MF}.

In this setting, the objective is to find a controller $K \!\in\! \mathcal{K}$ that minimizes the maximum expected regret among all joint probability distributions of the disturbances in an ambiguity set $\cal P$. This can be formulated formally as
\begin{equation}\label{eq::DROcost}
    \inf_{K\in\mathcal{K}}\sup_{P\in \mathcal{P}} \mathbb{E}_P [R(K,w,v)], 
\end{equation}
where the disturbances $\begin{bmatrix}
    w\\v
\end{bmatrix}$ are distributed according to $P\!\in\!\cal P$.

To solve this problem, we first need to characterize the ambiguity set $\mathcal{P}$ and explicitly determine a benchmark noncausal controller $K_0$.
As in \cite{DRORO}, we choose $\mathcal{P}$ to be the set of probability distributions that are at a distance of at most $r>0$ to a nominal probability distribution, $P_0\!\in\! {\mathcal{M}}({\mathbb{R}}^{N(n+p)})$. Here, the distance is chosen to be the type-2 Wasserstein distance defined as~\cite{wassOT}:
\begin{equation*}
    W_2^2(P_1,P_2):=\!\inf_{\pi\in\Pi(P_1,P_2)} \ \int_{\mathbb{R}^n\times\mathbb{R}^n} \|z_1\!-\!z_2 \|^2 \,\pi(dz_1,dz_2) ,
\end{equation*}
where the set $\Pi(P_1,P_2)$ comprises all joint distributions that have marginal distributions $P_1$ and $P_2$. Then, $\mathcal{P}$ can be written as: 
\begin{equation}
    \mathcal{P} := \{P \in \mathcal{M}(\mathbb{R}^{N(n+p)}) \,|\, W_2(P_0, P)\leq r\}.
\end{equation}

Unlike the full-information case, we know from Theorem 1 in \cite{ROMF} that in the measurement feedback case, there is no optimal noncausal controller that dominates every other controller for every disturbance. Therefore, we will choose $K_0$ as the optimal noncausal controller that minimizes the Frobenius norm of $T_K$. Theorem 3 in \cite{ROMF} shows that such a controller can be found as:
\begin{equation}
    K_0=(I+E_0J)^{-1} E_0,
\end{equation}
where the associated operator, $T_{K_0}$ is:
\begin{align}
    T_{K_0}=\begin{bmatrix}
        FE_0L+G&&FE_0\\ E_0L&&E_0
    \end{bmatrix},
    \end{align}
with 
    \begin{align}
          E_0 &\coloneqq -T^{-1}F^TGL^TU^{-1} \label{eq::E0}, \\
    T &\coloneqq I+F^TF \label{eq::T},\\
    U & \coloneqq I+LL^T\label{eq::U} . 
    \end{align}


\section{Tractable Formulation}\label{sec::cvx}
In this section, we introduce a tractable reformulation of the DR-RO-MF control problem~\eqref{eq::DROcost}.

\subsection{DR-RO-MF Control Problem}
Defining 
 \begin{equation}
     \mathcal{C}_K\coloneqq T_K^T T_K-T_{K_0}^T T_{K_0},
 \end{equation}
we can rewrite the DR-RO-MF control problem~\eqref{eq::DROcost} as
\begin{equation}\label{eq::DROcost_v2}
     \inf_{K\in\mathcal{K}}\sup_{P\in \mathcal{P}} \mathbb{E}_P \left[ \begin{bmatrix}
    w^T & v^T
\end{bmatrix} \cal C_K\begin{bmatrix}
    w \\
    v\\
\end{bmatrix} \right].
\end{equation}
The following theorem gives the dual problem of inner maximization and characterizes the worst-case distribution.

\begin{thm}\label{thm::SDQO}
    [adapted from 
Theorems 2 and 3 in~\cite{DRORO}]. Suppose $P_0$ is absolutely continuous with respect to the Lebesgue measure on $\mathbb{R}^{N}$ and  $\begin{bmatrix} w_0 \\ v_0 \end{bmatrix}\sim P_0$. The optimization problem:
\begin{equation}
    \sup_{P\in\mathcal{P}} \mathbb{E}_P\left[ \begin{bmatrix}
    w^T & v^T
\end{bmatrix} \cal C_K\begin{bmatrix}
    w \\
    v\\
\end{bmatrix} \right]
\end{equation}
where $\begin{bmatrix} w \\ v \end{bmatrix}\sim P$ and $\mathcal{C}_K\in\mathbb{S}^{N(n+p)}$, with $\lambda_{max}(\mathcal{C}_K)\neq 0$, has a finite solution and is equivalent to the convex optimization problem: 
\begin{equation}
    \inf_{\substack{\gamma \geq 0,\\ \gamma I \succ \mathcal{C}_K}} \gamma (r^2-\operatorname{Tr}(M_0)) + \gamma^2 \operatorname{Tr}(M_0(\gamma I-\mathcal{C}_K)^{-1}),
    \label{eq:opgamma}
\end{equation}
where $M_0:=\mathbb{E}_{P_0}\left[\begin{bmatrix}
    w \\
    v
\end{bmatrix} \begin{bmatrix}
    w^T&v^T
\end{bmatrix}\right]$. Furthermore, the disturbance that achieves the worst-case regret is $\begin{bmatrix} w^\ast \\ v^\ast \end{bmatrix} \sim P^\ast$, where $\begin{bmatrix} w^\ast \\ v^\ast \end{bmatrix} = \gamma^\ast (\gamma^\ast I - \mathcal{{C}}_{K})^{-1} \begin{bmatrix} w_0 \\ v_0 \end{bmatrix}$, and $\gamma^\ast$ is the optimal solution of~\eqref{eq:opgamma}, which also satisfies the algebraic equation: 
\begin{equation}
    \operatorname{Tr}( (\gamma(\gamma I - \mathcal{C}_K)^{-1}-I)^{2}M_0)=r^2
\end{equation}
\end{thm}
\begin{proof}
    The proof follows from Theorems 2 and 3 in \cite{DRORO} and is omitted for brevity here.
\end{proof}
    
We highlight two remarks pertaining to the presented theorem.

\noindent {\bf Remark 1:} Notice that the supremum of the quadratic cost depends on $P_0$ only though its covariance matrix $M_0$. Note further that as $r\rightarrow\infty$, the optimal $\gamma$ reaches its smallest possible value (since $r^2$ multiplies $\gamma$ in \eqref{eq:opgamma}). The smallest possible value that $\gamma$ can take is simply the operator norm of ${\cal C}_K$, which means that the DR-RO-MF controller approaches the regret-optimal controller as $r\rightarrow\infty$.

\noindent {\bf Remark 2:} Notice that the worst-case disturbance takes on a Gaussian distribution when the nominal disturbance is Gaussian. This is not immediately evident as the ambiguity set $\mathcal{P}$ contains non-Gaussian distributions. Note further that the worst-case disturbance is correlated even if the nominal distribution has white noise.

Assuming the covariance of the nominal distribution to be
\begin{equation}
M_0=\mathbb{E}_{P_0}\left[\begin{bmatrix}
    w \\
    v
\end{bmatrix} \begin{bmatrix}
    w^T&v^T
\end{bmatrix}\right]=I.\end{equation} 
so that $\operatorname{Tr}(M_0)=N(n+p)$, the optimization problem~\eqref{eq::DROcost_v2} can be cast equivalently using Theorem~\ref{thm::SDQO} as

\begin{equation}\label{newpb}
\begin{aligned}
&\inf_{K\in\mathcal{K}} \inf_{\substack{\gamma \geq 0}} \gamma (r^2-N(n+p)) + \gamma^2 \operatorname{Tr}((\gamma I - \mathcal{C}_K)^{-1})\\ 
&\text{s.t. } 
\begin{cases}
&\gamma I \succ \mathcal{C}_K\\
&\mathcal{C}_K=T_K^T T_K -T_{K_0}^T T_{K_0}\\
\end{cases}
\end{aligned}
\end{equation}
As in~\cite{ROMF}, define the unitary matrices $\Psi$ and $\Theta$: 
\begin{align}
    &\Theta=\begin{bmatrix}
        S^{-\frac{1}{2}}&&0\\0&& T^{-\frac{T}{2}}
    \end{bmatrix}\begin{bmatrix}
        I&&-F\\F^T&& I
    \end{bmatrix}\\
        &\Psi=\begin{bmatrix}
        I&&L^T\\-L&& I
    \end{bmatrix}\begin{bmatrix}
        V^{-\frac{1}{2}}&&-0\\0&& U^{-\frac{T}{2}}
    \end{bmatrix}
\end{align}
where $T$ and $U$ are as in~(\ref{eq::T}) and~(\ref{eq::U}), and
\begin{align}
    &S=I+FF^T\\
    &V=I+L^TL. \label{eq::V}
\end{align}
and $S^{\frac{1}{2}}$, $T^{\frac{1}{2}}$, $U^{\frac{1}{2}}$, and $V^{\frac{1}{2}}$ are (block) lower triangular matrices, such that $S=S^{\frac{1}{2}}S^{\frac{T}{2}}$, $T=T^{\frac{T}{2}}T^{\frac{1}{2}}$, $U=U^{\frac{1}{2}}U^{\frac{T}{2}}$, $V=V^{\frac{T}{2}}V^{\frac{1}{2}}$. Then, the optimization problem~\eqref{newpb} is equivalent to:
\begin{equation}\label{newpb2}
\begin{aligned}
&\inf_{\substack{K\in\mathcal{K}, \\ \gamma \geq 0,\\ \gamma I \succ \Bar{\mathcal{C}}_K}} \gamma (r^2-N(n+p)) + \gamma^2 \operatorname{Tr}((\gamma I - \Bar{\mathcal{C}}_K )^{-1}) \\
&\text{s.t. } 
\begin{cases}
& \Bar{\mathcal{C}}_K=(\Theta T_K \Psi)^T \Theta T_K \Psi-(\Theta T_{K_0} \Psi)^T \Theta T_{K_0} \Psi
\end{cases}
\end{aligned}
\end{equation}
which holds true since trace is invariant under unitary $\Theta$ and $\Psi$.
By introducing an auxiliary variable $X\succeq \gamma^2 (\gamma I - \Bar{\mathcal{C}}_K)^{-1}$ and leveraging the Schur complement theorem as in \cite{DRORO}, the problem \eqref{newpb2} can be recast as

\begin{equation}\label{eq::getX}
\begin{aligned}
&\inf_{\substack{K\in\mathcal{K},\\ \gamma \geq 0,\\ X \succeq 0}} \gamma (r^2-N(n+p)) + \operatorname{Tr}(X) \\
&\text{s.t. } 
\begin{cases}
&\begin{bmatrix}
    X & \gamma I \\
    \gamma I & \gamma I -  \Bar{\mathcal{C}}_K
\end{bmatrix} \succeq 0 \\
&\gamma I -  \Bar{\mathcal{C}}_K \succ 0\\
& \Bar{\mathcal{C}}_K\!=\!(\Theta T_K \Psi)^T \Theta T_K \Psi-(\Theta T_{K_0} \Psi)^T \Theta T_{K_0} \Psi
\end{cases}
\end{aligned}
\end{equation}

In the following lemma, we establish some of the important identities that are utilized to convert problem \eqref{eq::getX} to a tractable convex program.

\begin{lem}\label{thm::ROMF}
    [adapted from \cite{ROMF}]. The following statements hold:
    \begin{enumerate}
        \item  \textcolor{white}{.}
        \vspace{-5mm}
        \begin{align}\label{eq::gammaI}
       &\gamma I - \Bar{\mathcal{C}}_K
       =\begin{bmatrix}
           \gamma I && -PZ \\
           -Z^T P^T&&\gamma I -Z^TZ    
       \end{bmatrix}
    \end{align} 
    where
\begin{align}
    &Z =T^{\frac{1}{2}}EU^{\frac{1}{2}}-W \label{eq::Z}\\
    &W =-T^{-\frac{T}{2}}F^TGL^TU^{-\frac{T}{2}}\label{eq::W}\\
    &P =V^{-\frac{T}{2}}G^TFT^{-\frac{1}{2}} \label{eq::P}
\end{align}
 and $E$, $T$, $U$ and $V$ are as defined in~\ref{eq::E1},~\ref{eq::T}, ~\ref{eq::U} and~\ref{eq::V} respectively.

\item\textcolor{white}{.}
        \vspace{-5mm}
\begin{equation}\label{eq::nehari}
   \gamma I -  \Bar{\mathcal{C}}_K \succ 0 \Leftrightarrow \| Y - W_{-,\gamma}\|_{2} \leq 1
\end{equation}
where 
\begin{align}
&\gamma^{-1} I+ \gamma^{-2} P^TP= M_{\gamma}^T M_{\gamma}\\
 &M_{\gamma} = \left(\gamma^{-1} I+ \gamma^{-2} P^TP\right)^{\frac{1}{2}}\label{eq::Mgamma}\\
&W_{\gamma} =M_{\gamma}W\\
    &Y =M_{\gamma} T^{\frac{1}{2}} EU^{\frac{1}{2}} - W_{+,\gamma}
\end{align}
and $W_{+,\gamma}$ and $W_{-,\gamma}$ are the causal and strictly anticausal parts of $W_{\gamma}$. Here, $M_{\gamma}$ is lower triangular, and positive-definite.

\item $Y$ is causal iff $E$ is causal, where $E$ can be found as follows: 
\begin{equation}
    E=T^{-\frac{1}{2}}M_{\gamma}^{-1}(Y+W_{+,\gamma})U^{-\frac{1}{2}}\label{eq::Enew}
\end{equation}

\item The condition in~(\ref{eq::nehari}) is recognized as a level-1 suboptimal Nehari problem that approximates a strictly anticausal matrix $W_{-,\gamma}$ by a causal matrix $Y$.
\end{enumerate}
\end{lem}

\begin{proof}
    The proof follows from Theorem 4 in \cite{ROMF} and is omitted for brevity here.
\end{proof}

Using Lemma~\ref{thm::ROMF}, problem~\eqref{eq::getX} can be reformulated as a tractable optimization program:
\begin{equation}
    \begin{aligned}
    &\inf_{\substack{Z,Y\in\mathcal{K},\\ \gamma \geq 0,\\ X \succeq 0}} \gamma (r^2-N(n+p)) + \operatorname{Tr}(X) \\
    &\text{s.t. } 
    \begin{cases}
    &\begin{bmatrix}
        X_{11} &X_{12} & \gamma I & 0\\
          X_{12}^T &X_{22} & 0& \gamma I \\
        \gamma I & 0 &\gamma I & -PZ \\         
        0 & \gamma I &  -Z^T P^T&\gamma I -Z^TZ 
    \end{bmatrix}\succeq 0 \\
    &\| Y - W_{-,\gamma}\|_{2} \leq 1          
    \end{cases}
\end{aligned}
\end{equation}

\begin{equation}
\begin{aligned}
    &=\inf_{\substack{Z,Y\in\mathcal{K},\\ \gamma \geq 0,\\ X \succeq 0}} \gamma (r^2-N(n+p)) + \operatorname{Tr}(X) \\
    &\text{s.t. } 
    \begin{cases}
    &\begin{bmatrix}
        X_{11} &X_{12} & \gamma I & 0 & 0\\
          X_{12}^T &X_{22} & 0& \gamma I & 0 \\
        \gamma I & 0 &\gamma I & -PZ & 0 \\         
        0 & \gamma I &  -Z^T P^T&\gamma I & Z^T\\
        0&0&0&Z&I
    \end{bmatrix}\succeq 0 \label{case::1}\\
    &\| Y - W_{-,\gamma}\|_{2} \leq 1   
    \end{cases}
\end{aligned}
\end{equation}
where the last step follows from the Schur complement. Using~\eqref{eq::Z},~\eqref{eq::Enew}, and 
\begin{equation}
    H_\gamma=M_\gamma^{-1}W_{+,\gamma}-W
\end{equation}
we establish our main theorem.
\begin{thm}[Tractable Formulation of DR-RO-MF]
The distributionally robust regret-optimal control problem in the measurement feedback setting~\eqref{eq::DROcost} reads:
\textcolor{white}{ }
\begin{equation}\begin{aligned}\label{eq::finalpb}
&\inf_{\substack{Y\in\mathcal{K},\\ \gamma \geq 0,\\ X \succeq 0}} \gamma (r^2-N(n+p)) + \operatorname{Tr}(X) \\
    &\text{s.t. } 
    \begin{cases}
    &\begin{bmatrix}
        X_{11} &X_{12} \!&\! \gamma I & 0 & 0\\
          X_{12}^T &X_{22} \!&\! 0& \gamma I & 0 \\
        \gamma I & 0 \!&\!\gamma I & -P(*) & 0 \\         
        0 & \gamma I \!&\!  -(*)^T P^T&\gamma I & (*)^T\\
        0&0&0&(*)&I
    \end{bmatrix}\!\succeq\! 0 \\
     &(*)=M_\gamma^{-1}Y+H_\gamma\\
    &\begin{bmatrix}
        I& (Y - W_{-,\gamma})^T\\
        Y - W_{-,\gamma}& I
    \end{bmatrix}\succ 0
    \end{cases}
\end{aligned}\end{equation}
The optimal controller $K^\ast$ is then obtained using~\eqref{eq::controller} and~\eqref{eq::Enew}.
\end{thm}

\subsection{Sub-Optimal Problem}\label{subsec::subopt}
For a given value of $\gamma$, problem~(\ref{eq::finalpb}) can be simplified into a tractable SDP. In practical implementations, we can solve problem~(\ref{eq::finalpb}) by optimizing the objective function with respect to the variables $Y$ and $X$ while fixing $\gamma$, thus transforming the problem into an SDP, which can be solved using standard convex optimization packages. We then iteratively refine the value of $\gamma$ until it converges to the optimal solution $\gamma^*$. This iterative process ensures that we obtain the best possible value for $\gamma$ that minimizes the objective function in problem~(\ref{eq::finalpb}).


\subsection{\texorpdfstring{LQG and RO-MF Control Problems as Special Cases}{LQG and RO-MF Control Problems as Special Cases}}\label{sub::special}

Interestingly, LQG and RO control in the measurement feedback setting can be recovered from the DR-RO-MF control by varying the radius $r$ which represents the extent of uncertainty regarding the accuracy of the nominal distribution in the ambiguity set. When $r\rightarrow 0$, the ambiguity set transforms into a singular set comprising solely the nominal distribution. Consequently, the problem simplifies into a stochastic optimal control problem under partial observability:
\begin{equation}
\inf_{K\in \mathcal{K}} \mathbb{E}_{P_0} [J(K,w,v)]
\end{equation}
As $r\rightarrow \infty$,  the ambiguity set transforms into the set of any disturbance generated adversarially and the optimal $\gamma$ reaches its smallest possible value which is the operator norm of ${\cal C}_K$. This means that the problem reduces to the RO-MF control problem which we discussed in section~\ref{sec::ROMF}.

\section{Simulations}\label{sec::exp}
\subsection{Flight Control}\label{subsec::flight}
We focus on the problem of controlling the longitudinal flight of a Boeing 747 which pertains to the linearized dynamics of the aircraft, as presented in~\cite{boeing}. The linear dynamical system provided describes the aircraft's dynamics during level flight at an altitude of 7.57 miles and a speed of 593 miles per hour, with a discretization interval of 0.1 second. The state variables of the system encompass the aircraft's velocity along the body axis, velocity perpendicular to the body axis, angle between the body axis and the horizontal plane, and angular velocity. The inputs to the system are the elevator angle and thrust. The process noise accounts for variations caused by external wind conditions.
The discrete-time state space model is:
\begin{align*}
    &A= \begin{bmatrix}
      0.9801& 0.0003& -0.0980& 0.0038\\
           -0.3868& 0.9071& 0.0471& -0.0008\\
           0.1591& -0.0015& 0.9691& 0.0003\\
           -0.0198& 0.0958& 0.0021& 1.000
\end{bmatrix}\\
       &B= \begin{bmatrix}
           -0.0001&0.0058\\
            0.0296& 0.0153\\
            0.0012& -0.0908\\
            0.0015& 0.0008
       \end{bmatrix},
        C=\begin{bmatrix}
            1&0&0&0\\
            0&0&0&1
        \end{bmatrix}.
\end{align*}
We conduct all experiments using MATLAB, on a PC with an Intel Core i7-1065G7 processor and 16 GB of RAM. The optimization problems are solved using the CVX package~\cite{cvx}.

We limit the horizon to $N=10$. We take the nominal distribution $P_0$ to be Gaussian with mean $\mu_0=0$ and covariance $\Sigma_0=I$, and we investigate various values for the radius $r$, specifically:
\begin{equation*}
    r\in \{0, 0.2, 0.4, 0.6, 0.8, 1, 1.5, 2, 4, 8, 16, 32, 126\}.
\end{equation*}

For each value of $r$, we solve the sub-optimal problem described in section~\ref{subsec::subopt}, iterating over $\gamma$ until convergence to $\gamma^*$.

\begin{figure}[!ht]
    \centering    
    \includegraphics[width=0.52\textwidth]{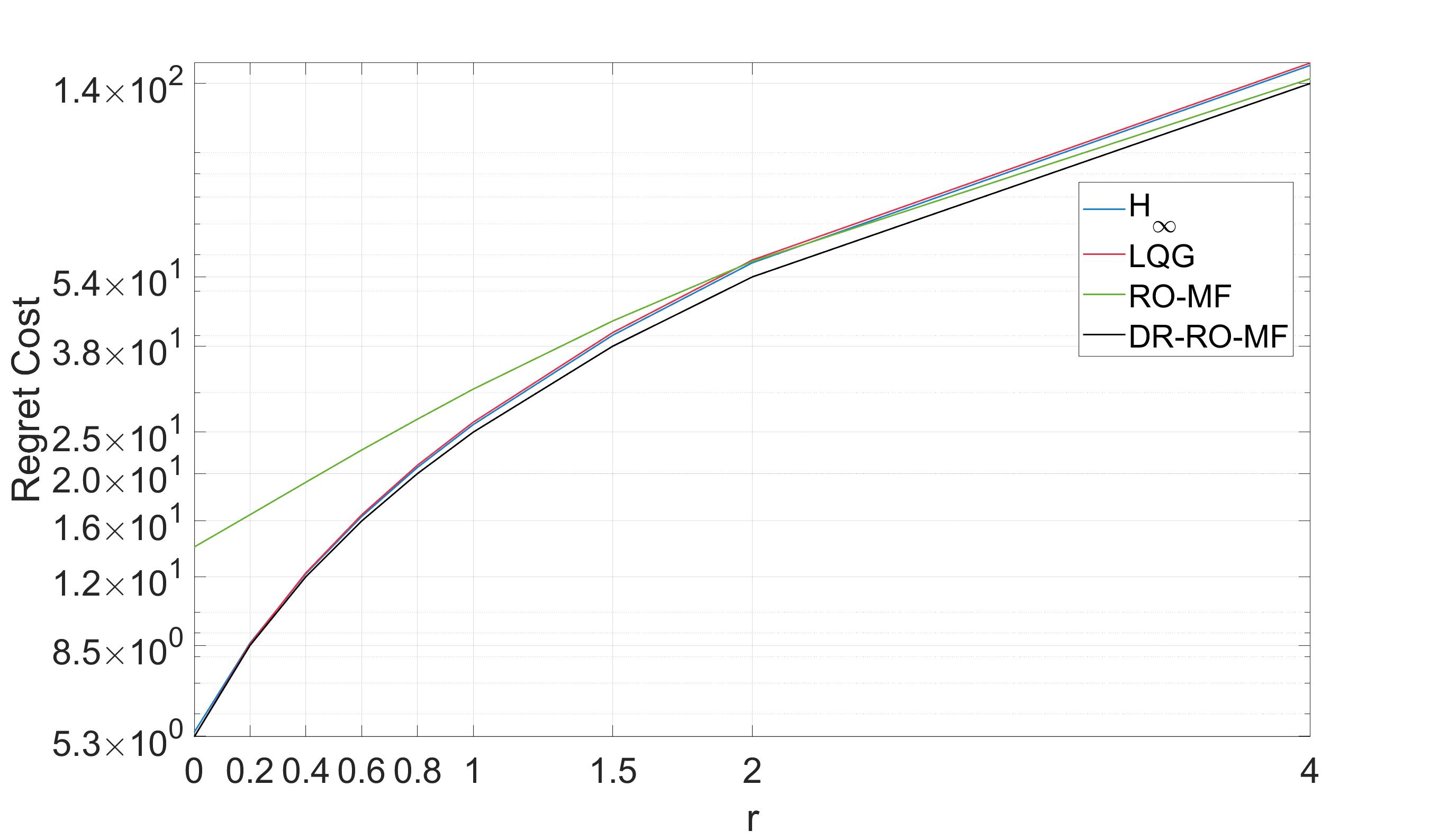}
    \caption{Controller costs for $r\in {0, 0.2, 0.4, 0.6, 0.8, 1, 1.5, 2, 4}$.\\
    At $r=0$, the top-performing controllers are DR-RO-MF and LQG, exhibiting regret costs of 5.4. They are followed by $\text{H}_\infty$ with a regret cost of 5.9, and finally RO-MF with a regret cost of 13.8. The ranking of the controllers based on regret costs is: \textbf{DR-RO-MF}=$\textbf{LQG}$=5.34 $<$ $\textbf{H}_\infty$=5.47 $<$ \textbf{RO-MF}=13.8. \\
    As $r$ increases to 4, DR-RO-MF remains the best-performing controller with a regret of 141. It is followed by RO-MF with a regret of 144, $\text{H}_\infty$ with a regret of 154, and finally $\text{H}2$ with a regret of 156. The ranking of the controllers at $r=4$ based on regret costs is: \textbf{DR-RO-MF}=141 $<$ \textbf{RO-MF}=144 $<$ $\textbf{H}_\infty$=154 $<$ \textbf{LQG}=156.}
    \label{fig:small_r}
\end{figure}

\begin{figure}[!ht]
    \centering
    \includegraphics[width=0.52\textwidth]{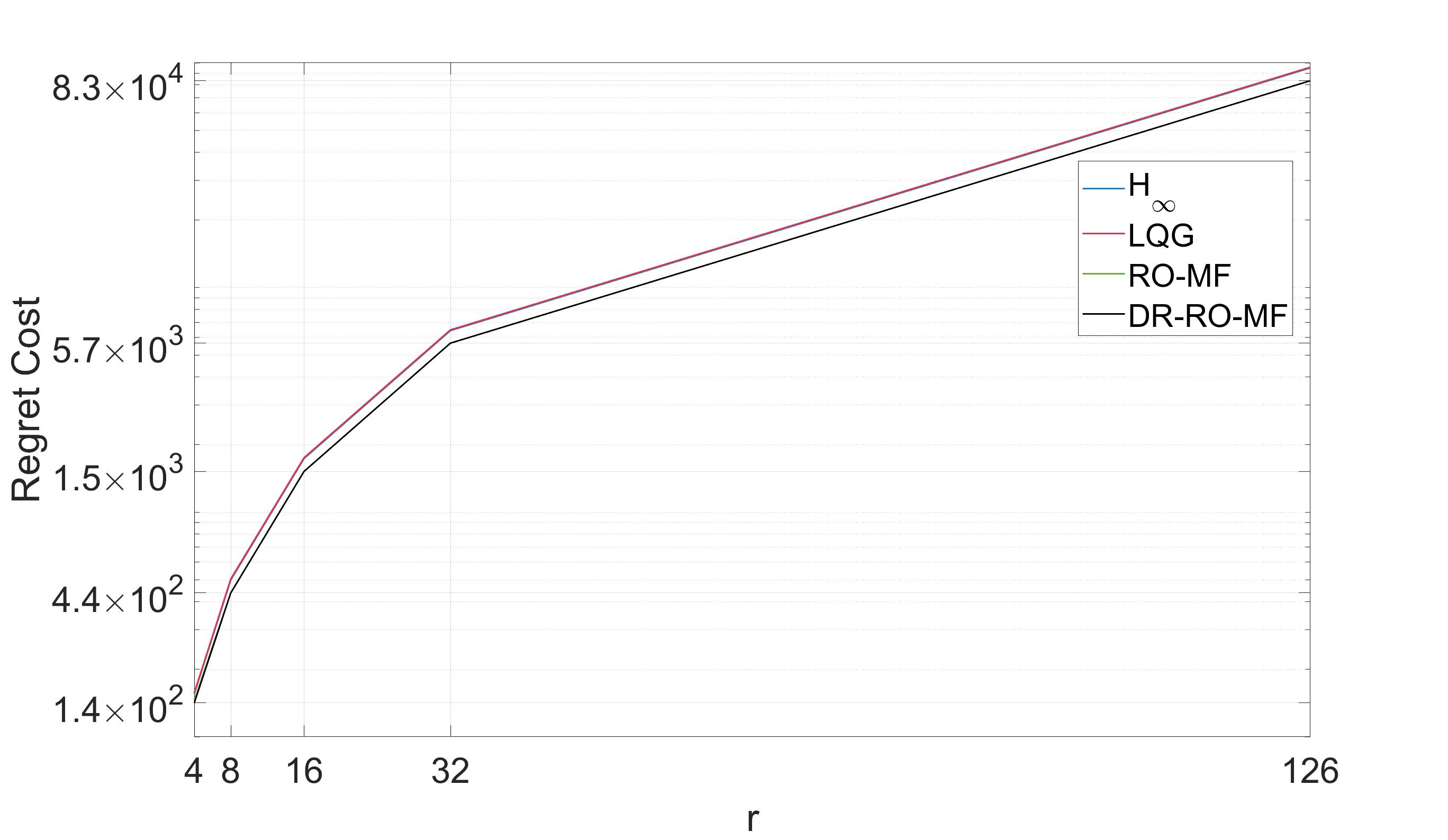}
    \caption{Controller costs for $r\in {4, 8, 16, 32, 126}$.\\
At $r=8$, the best-performing controller is DR-RO-MF with a regret of 437, which is closely comparable to the regret of the RO-MF controller of 438. They are followed by $\text{H}_\infty$ with a regret of 499, and finally $\text{H}2$ with a regret of 505. The ranking of controllers based on regret costs is as follows: \textbf{DR-RO-MF}=437 $\lesssim$ \textbf{RO-MF}=438 $<$ $\textbf{H}_\infty$=499 $<$ \textbf{LQG}=505.\\
When $r$ increases to 126, which approximates the behavior of $r$ approaching infinity, the order of the best-performing controllers remains unchanged: \textbf{DR-RO-MF}=\textbf{RO-MF}=$8.33 \times 10^4$ $<$ $\textbf{H}_\infty$=$9.50 \times 10^4$ $<$ \textbf{LQG}=$9.57 \times 10^4$. DR-RO-MF and RO-MF controllers exhibit similar performance in this regime.
    }
    \label{fig:big_r}
\end{figure}

To assess the performance of the controller, we compute the worst-case disturbance, which lies at a Wasserstein distance $r$ from $P_0$, as discussed in theorem~\ref{thm::SDQO}. Finally, we compare the regret cost of the DR-RO-MF controller with that of the LQG, $H_\infty$~\cite{blackbook}, and RO-MF~\cite{ROMF} controllers while considering the worst-case disturbance corresponding to the DR-RO-MF controller. The results are shown in Figures~\ref{fig:small_r} and \ref{fig:big_r}.

The DR-RO-MF controller achieves the minimum cost under worst-case disturbance conditions for any given value of $r$. When $r$ is sufficiently small (less than 0.2), the cost of the DR-RO-MF controller closely approximates that of the LQG controller (figure~\ref{fig:small_r}). Conversely, for sufficiently large values of $r$ (greater than 8), the cost of the DR-RO-MF controller closely matches that of the RO-MF controller (figure~\ref{fig:big_r}). These observations align with theoretical findings as elaborated in section~\ref{sub::special}.

Furthermore, it is worth noting that for large values of $r$ (figure~\ref{fig:big_r}), the LQG controller yields the poorest results. Conversely, for small values of $r$ (figure~\ref{fig:small_r}), the LQG controller performs on par with the DR-RO-MF controller, emerging as the best choice, as mentioned earlier. This discrepancy is expected since LQG control accounts only for disturbances drawn from the nominal distribution, assuming uncorrelated noise. On the other hand, RO-MF exhibits inferior performance when $r$ is small (figure~\ref{fig:small_r}), but gradually becomes the top-performing controller alongside DR-RO-MF as $r$ increases. This behavior arises from the fact that RO-MF is specifically designed for sufficiently large $r$. Lastly, note that the $H_\infty$ cost lies between the costs of the other controllers, interpolating their respective costs.

\subsection{Performance Under Adversarially Chosen Distribution}
For any given causal controller $K_{c}$, an adversary can choose the worst-case distribution of disturbances for a fixed $r$ as
\begin{align}
\arg \max_{P\in \mathcal{P}} \mathbb{E}_{ P} R(K_{c},w,v) \eqqcolon P_{c}, \label{eq:adversarial}
\end{align}
where $R$ is the regret as in \eqref{eq::regret}. Denoting by $K_{\text{DR-RO-MF}}$ the optimal DR-RO-MF controller and by $P_{\text{DR-RO-MF}}$ the worst-case (adversarial) distribution corresponding to $K_{\text{DR-RO-MF}}$, we have that
\begin{align}
     \mathbb{E}_{ P_c} R(K_{c},w,v) &= \max_{P\in\mathcal P}\mathbb{E}_{P} R(K_{c},w,v), \\
     &\geq \min_{K\in \mathcal K}\max_{P\in\mathcal P}\mathbb{E}_{P} R(K,w,v), \\
     &=  \mathbb{E}_{ P_{\text{DR-RO-MF}}} R(K_{\text{DR-RO-MF}},w,v),  \\
     &\geq \mathbb{E}_{ P_{c}} R(K_{\text{DR-RO-MF}},w,v), \label{eq:theoretical_exp}
\end{align}
where the first equality follows from \eqref{eq:adversarial} and the last inequality is due to the fact that $P_{\text{DR-RO-MF}}$ is the worst-case distribution for $K_{\text{DR-RO-MF}}$.  In other words, DR-RO-MF controller is robust to adversarial changes in distribution as it yields smaller expected regret compared to any other causal controller $K_c$ when the disturbances are sampled from the worst-case distribution $P_c$ corresponding to $K_c$.

The simulation results presented in Subsection \ref{subsec::flight} show that DR-RO-MF outperforms  RO-MF, $H_\infty$, and LQG (designed assuming disturbances are sampled from $P_0$) controllers under the worst-case distribution of the DR-RO-MF controller $P_{\text{DR-RO-MF}}$, i.e
\begin{align}
     \mathbb{E}_{ P_\text{DR-RO-MF}} R(K_{c},w,v) \geq \mathbb{E}_{ P_{\text{DR-RO-MF}}} R(K_{\text{DR-RO-MF}},w,v).
\end{align}
This directly implies that the theoretically expected inequality
\begin{align}
    \mathbb{E}_{ P_c} R(K_{c},w,v) \geq \mathbb{E}_{ P_{c}} R(K_{\text{DR-RO-MF}},w,v) \label{eq:gap}
\end{align}
is validated and positively exceeded following the inequalities \eqref{eq:theoretical_exp} and
\begin{align}
     &\mathbb{E}_{ P_c} R(K_{c},w,v) \geq \mathbb{E}_{ P_{\text{DR-RO-MF}}} R(K_{c},w,v).
\end{align}

To further support our claims, we assess the performance of LQG and RO-MF controllers by measuring the relative reduction in expected regret when DR-RO-MF controller is utilized under the worst-case distributions corresponding to LQG and RO-MF controllers, respectively:
 \begin{equation}\label{eq::perc}
     \frac{\mathbb{E}_{P_c} R(K_{c},w,v) - \mathbb{E}_ {P_c} R(K_{\text{DR-RO-MF}},w,v)}{ \mathbb{E}_{P_c} R(K_{c},w,v)} \times 100,
 \end{equation} 
 where $K_c$ is either LQG or RO-MF controller and $P_c$ is the corresponding worst-case distribution. The results are shown in Table~\ref{tab:my_label} for $r \in \{0.2,1,2,4,16,32\}$.

 \begin{table}[]
     \centering
     \begin{tabular}{|c||c|c|c|c|c|c|}
    \hline
    $\mathbf{r}$ & \textbf{0.2} & \textbf{1} & \textbf{2} & \textbf{4}  & \textbf{16} & \textbf{32} \\
    \hline\hline
    \textbf{LQG}(\%) & 0.860 & 8.17 & 14.8 &21.9& 28.5 & 29.3 \\
    \hline
    \textbf{RO-MF}(\%) & 56.6 & 43.0 & 32.3 & 17.2 & 1.95 & 0.465 \\
    \hline
    \end{tabular}
    \vspace{1em} 
     \caption{Relative difference in \% (as in \eqref{eq::perc}) between the expected regret of LQG/RO-MF and of DR-RO-MF controllers, under the worst-case disturbance of LQG/RO-MF, respectively, as in \eqref{eq:adversarial} for different values of $r$} 
     \label{tab:my_label}
 \end{table}

\subsection{Limitations}
In our scenario with a relatively short planning horizon of $N=10$, the cost reduction achieved by employing DR-RO-MF control, in comparison to traditional controllers such as LQG and $H_\infty$, is moderate. However, it is anticipated that this reduction would become more pronounced with the utilization of a longer planning horizon. Unfortunately, in our experimental setup, we were restricted to using $N=10$ due to computational limitations. Solving semi-definite programs involving large matrices is computationally inefficient, necessitating this constraint. In practice, this limitation can be overcome by implementing the controller in a receding horizon fashion, where the controller is updated every $x$ time steps.

\section{Conclusion}
In conclusion, this paper extended the distributionally robust approach to regret-optimal control by incorporating the Wasserstein-2 distance~\cite{DRORO} to handle cases of limited observability. The proposed DR-RO-MF controller demonstrated superior performance compared to classical controllers such as LQG and $H_\infty$, as well as the RO-MF controller, in simulations of flight control scenarios. The controller exhibits a unique interpolation behavior between LQG and RO-MF, determined by the radius $r$ that quantifies the uncertainty in the accuracy of the nominal distribution. As the time horizon increases, solving the tractable SDP to which the solution reduces, becomes more challenging, highlighting the practical need for a model predictive control approach. Overall, the extended distributionally robust approach presented in this paper holds promise for robust and effective control in systems with limited observability.

\bibliographystyle{./bibliography/IEEEtran}
\bibliography{./bibliography/IEEEabrv,./bibliography/IEEEexample}

\vspace{12pt}

\end{document}